\documentclass[12pt]{article}
\usepackage{amsmath,amsthm,amsfonts,latexsym,amsopn,verbatim,amscd,amssymb}
\theoremstyle{plain}
\newtheorem{theorem}{Theorem}[section]
\newtheorem{lemma}[theorem]{Lemma}

\newcommand{\bnum}{\begin{enumerate}}
\newcommand{\enum}{\end{enumerate}}

\numberwithin{equation}{section}

\begin{document}

\title{\textbf{Some bounds for commuting probability of finite rings}}
\author{Jutirekha Dutta, Dhiren Kumar Basnet\footnote{Corresponding author}}
\date{}
\maketitle
\begin{center}\small{\it 
Department of Mathematical Sciences, Tezpur University,\\ Napaam-784028, Sonitpur, Assam, India.\\



Emails:\, jutirekhadutta@yahoo.com and  dbasnet@tezu.ernet.in}
\end{center}

\medskip

\begin{abstract} 
Let $R$ be a finite ring. The commuting probability of  $R$ is the probability that any two randomly chosen elements of $R$ commute. In this paper, we obtain some bounds for commuting probability of $R$.
\end{abstract}

\medskip

 

\medskip

\section{Introduction}
Throughout the paper $R$ denotes a finite ring. The commuting probability of $R$, denoted by $\Pr(R)$, is the probability that a randomly chosen pair of elements of $R$ commute. That is 
\begin{equation}\label{com_prob_eq1}
\Pr(R) = \frac{|\{(s, r) \in R \times R : sr = rs\}|}{|R \times R|}.
\end{equation}
The study of $\Pr(R)$ was initiated  by MacHale \cite{dmachale} in the year 1976. After the  works of Erd$\ddot{\rm o}$s and Tur$\acute{\rm a}$n \cite{pEpT68}, many papers have been written  on commuting probability of finite groups in the last few decades, for example see \cite{DN11} and the references therein.
 However, people did not work much on commuting probability of finite rings. We have only few papers \cite{BM, BMS, dmachale} on $\Pr(R)$ in the literature. 
In this paper, we obtain some bounds for $\Pr(R)$. 


Recall that $[s, r]$ to denote the additive commutator $sr - rs$ for any two elements $s, r \in R$.  By $K(R, R)$ we denote the set $\{[s, r] : s, r \in R\}$ and $[R, R]$ denotes the subgroup of $(R, +)$ generated by $K(R, R)$. Note that $[R, R]$ is the commutator subgroup of $(R, +)$ (see \cite{BMS}). Also, for any $x \in R$, we write $[x, R]$ to denote the subgroup of $(R, +)$ consisting of all elements of the form $[x, y]$ where $y \in R$.

\section{Main Results}

Let $C_R(r)$ denote the subset  $\{s \in R : sr = rs\}$ of $R$, where $r$ is  an element of $R$. Then $C_R(r)$ is a subring of $R$  known as centralizer of $r$ in $R$. Note that the center $Z(R)$ of $R$ is the intersection of all the centralizers in $R$. 

By \eqref{com_prob_eq1}, we have
\[
\Pr(R) = \frac{1}{|R|^2} \underset{r \in R}{\sum}|C_R(r)|
\]
and hence
\begin{equation}\label{com_prob_eq2}
\Pr(R) = \frac{|Z(R)|}{|R|}  + \frac{1}{|R|^2} \underset{r \in R \setminus Z(R)}{\sum}|C_R(r)|.
\end{equation}

If $p$ is the smallest prime dividing the order of a finite non-commutative ring $R$ then, by  \cite[Theorem 2]{dmachale}, we have
\begin{equation}\label{com_prob_eq3}
\Pr(R) \leq \frac{p^2 + p - 1}{p^3}.
\end{equation}
In the following theorem, we  give two bounds for $\Pr(R)$. We shall see that  the upper bound for $\Pr(R)$ in Theorem \ref{theorem001} is better than \eqref{com_prob_eq3}. 
\begin{theorem}\label{theorem001}
Let $R$ be a finite non-commutative ring. If $p$ is the smallest prime dividing $|R|$ then 
\begin{enumerate}
\item[\rm (a)] $\Pr(R) \geq \frac{|Z(R)|}{|R|} + \frac{p(|R| - |Z(R)|)}{|R|^2}$ with equality if and only if $|C_R(r)| = p$ for all $r \notin Z(R)$.
\item[\rm (b)] $\Pr(R) \leq \frac{(p - 1)|Z(R)| + |R|}{p|R|}$ 
with equality if and only if $|R: C_R(r)| = p$ for all $r \notin Z(R)$.
\end{enumerate}
\end{theorem}
\begin{proof}
By \eqref{com_prob_eq2}, we have 
\begin{equation}\label{boundeq-3}
|R|^2 \Pr(R) =  |R||Z(R)| + \underset{r \in R\setminus Z(R)}{\sum}|C_R(r)|.
\end{equation}
(a) If $r \notin Z(R)$ then $|C_R(r)| \geq p$. Therefore
\[
\underset{r \in R\setminus Z(R)}{\sum}|C_R(r)| \geq p(|R| - |Z(R)|) 
\]
with equality if and only if $|C_R(r)| = p$ for all $r \notin Z(R)$. Hence, the result follows from \eqref{boundeq-3}.

(b) If $r \notin Z(R)$ then $|C_R(r)| \leq \frac{|R|}{p}$. Therefore
\[
\underset{r \in R\setminus Z(R)}{\sum}|C_R(r)| \leq \frac{|R|(|R| - |Z(R)|)}{p}
\]
with equality if and only if $|R: C_R(r)| = p$ for all $r \notin Z(R)$. Hence, the result follows from \eqref{boundeq-3}.
\end{proof}

\noindent If $R$ is a non-commutative ring and $p$  the smallest prime dividing $|R|$ then $|R : Z(R)| \geq p^2$. Therefore
\[
\frac{(p - 1)|Z(R)| + |R|}{p|R|} \leq \frac{p^2 + p - 1}{p^3}.
\]
Thus the  bound obtained in Theorem  \ref{theorem001}(b) is better than \eqref{com_prob_eq3}.


If $S$ is a subring of $R$ then  MacHale \cite[Theorem 4]{dmachale} showed that
\begin{equation}\label{MachaleThm4}
\Pr(R) \leq \Pr(S).
\end{equation} 
We now proceed to derive an improvement of \eqref{MachaleThm4}. We require the following lemma. 
\begin{lemma}\label{lemma2}
Let  $N$ be an ideal of a finite non-commutative ring $R$. Then 
\[
\frac{C_R(x) + N}{N} \subseteq C_{R/N}(x + N) \; \text{for all} \; x \in R.
\]
The equality holds if $N \cap [R, R] = \{0\}$.
\end{lemma}

\begin{proof}
For any element $s \in C_R(x) + N$, where $s = r + n$ for some $r \in C_R(x)$ and $n \in N$, we have $s + N = r + N \in R/N$. Also, 
\[
(s + N)(x + N) = rx + N = xr + N = (x + N)(s + N), 
\]
as $r \in C_R(x)$. This proves the first part.

Let $N \cap [R, R] = \{0\}$ and $y + N \in C_{R/N}(x + N)$. Then $y \in R$ and  $(y + N)(x + N) = (x + N)(y + N)$. This gives $yx - xy \in N \cap [R, R] = \{0\}$ and so $y \in C_R(x)$. Therefore, $y + N \in \frac{C_R(x) + N}{N}$. Hence the equality holds. 
\end{proof}
The following result which is an improvement of \eqref{MachaleThm4} also gives a relation between $\Pr(R), \Pr(R/N)$ and $\Pr(N)$, where $N$ is an ideal of $R$.
\begin{theorem}\label{theorem3}
Let  $N$ be an ideal of a finite non-commutative finite ring $R$. Then 
\[
\Pr(R) \leq \Pr (R/N) \Pr(N).
\]
The equality holds if  $N \cap [R, R] = \{0\}$.
\end{theorem}

\begin{proof}
We have that 
\allowdisplaybreaks{
\begin{align*}
|R|^2 \Pr(R) = &\underset{x \in R}{\sum} |C_R(x)| \\
= &\underset{S \in \frac{R}{N}}{\sum}\underset{y \in S}{\sum}\frac{|C_R(y)|}{|N \cap C_R(y)|} |C_N(y)| \\
= &\underset{S \in \frac{R}{N}}{\sum}\underset{y \in S}{\sum}\frac{|C_R(y) + N|}{|N|} |C_N(y)| \\
\leq &\underset{S \in \frac{R}{N}}{\sum}\underset{y \in S}{\sum}|C_{\frac{R}{N}}(y + N)| |C_N(y)| \;\;(\text{using Lemma \ref{lemma2}}) \\
= &\underset{S \in \frac{R}{N}}{\sum} |C_{\frac{R}{N}}(S)| \underset{y \in S}{\sum}|C_N(y)| \\
= &\underset{S \in \frac{R}{N}}{\sum} |C_{\frac{R}{N}}(S)| \underset{n \in N}{\sum}|C_R(n) \cap S|.  
\end{align*}
} 
Let $a + N = S$ where $a \in R\setminus N$. If $C_R(n) \cap S = \phi$ then $|C_R(n) \cap S| < |C_N(n)|$. If $C_R(n) \cap S \neq \phi$ then there exists $x_0 \in C_R(n) \cap S$ such that $x_0 = a + n_0$ for some $a \in R\setminus N$ and $n_0 \in N$. Therefore $x_0 + N = a + N = S$ and so $S \cap C_R(n) = (x_0 + N) \cap (x_0 + C_R(n)) = x_0 + (N \cap C_R(n)) = x_0 + C_N(n)$. Hence $|S \cap C_R(n)| \leq |C_N(n)|$. This gives
\begin{align*}
|R|^2 \Pr(R) \leq &\underset{S \in \frac{R}{N}}{\sum} |C_{\frac{R}{N}}(S)| \underset{n \in N}{\sum}|C_N(n)| \\
= & |R/N|^2 \Pr(R/N) |N|^2 \Pr(N) \\
= & |R|^2 \Pr(R/N)\Pr(N).  
\end{align*}
Hence the bound follows.

 Let  $N \cap [R, R] = \{0\}$. Then, by Lemma \ref{lemma2}, we have  
 \[
\frac{C_R(x) + N}{N} = C_{R/N}(x + N) \; \text{for all} \; x \in R.
\]
If $S = a + N$ then it can be seen that $a + n \in C_R(n) \cap S$ for all $n \in N$. Therefore, $C_R(n) \cap S \ne \phi$ for all $n \in N$ and for all $S \in R/N$. Thus all the inequalities above become equalities if $N \cap [R, R] = \{0\}$. This completes the proof.
\end{proof}

In Lemma 2.3 of \cite{BMS}, Buckley et al. showed that
\begin{equation}\label{buckley_eq}
\Pr(R) > \frac{1}{|[R, R]|}.
\end{equation}
The following two results give some improvements of \eqref{buckley_eq}.

\begin{theorem}\label{newlb1}
Let $R$ be a  finite   ring  $R$. Then 
\[
\Pr(R) \geq \frac{1}{|K(S, R)|}\left(1 + \frac{|K(R, R)| - 1}{|R : Z(R)|} \right)
\]
with equality if and only if $|K(R, R)| = |[r, R]|$ for all $r \in R\setminus Z(R)$. In particular, if $R$ is non-commutative then $\Pr(R) > \frac{1}{|K(R, R)|}$.
\end{theorem}
\begin{proof}
By   \eqref{com_prob_eq2}, we have
\[
\Pr(R) =   \frac{|Z(R)|}{|R|}   + \frac{1}{|R|}\underset{r \in R\setminus Z(R)}{\sum}\frac{1}{|R : C_R(r)|}.
\]
Since $|K(R, R)| \geq |[r, R]| = |R : C_R(r)|$ for all $r \in R\setminus Z(R)$, we have
\begin{align*}
\Pr(R) \geq & \frac{|Z(R)|}{|R|}   + \frac{1}{|R|} \underset{r \in R\setminus Z(R)}{\sum}\frac{1}{|K(R, R)|}\\
= & \frac{|Z(R)|}{|R|}   +  \frac{|R| - |Z(R)|}{|R||K(R, R)|}
\end{align*}
from which the result follows.
\end{proof}
\noindent We also have the following lower bound.
\begin{theorem}\label{newlb2}
Let $S$ be a subring of a finite   ring  $R$. Then 
\[
\Pr(R) \geq \frac{1}{|[R, R]|}\left(1 + \frac{|[R, R]| - 1}{|R : Z(R)|} \right)
\]
with equality if and only if $|[R, R]| = |[r, R]|$ for all $r \in R\setminus Z(R)$. In particular, if $R$ is non-commutative then $\Pr(R) > \frac{1}{|[R, R]|}$.
\end{theorem}
\begin{proof}
By   \eqref{com_prob_eq2}, we have
\[
\Pr(R) =   \frac{|Z(R)|}{|R|}   + \frac{1}{|R|}\underset{r \in R\setminus Z(R)}{\sum}\frac{1}{|R : C_R(r)|}.
\]
Since $|[R, R]| \geq |R : C_R(r)|$ for all $r \in R\setminus Z(R)$, we have
\begin{align*}
\Pr(R) \geq & \frac{|Z(R)|}{|R|}   + \frac{1}{|R|} \underset{r \in R\setminus Z(R)}{\sum}\frac{1}{|[R, R]|}\\
= & \frac{|Z(R)|}{|R|}   +  \frac{|R| - |Z(R)|}{|R||[R, R]|}
\end{align*}
from which the result follows.
\end{proof}

Let $p$ be the smallest prime dividing $|R|$. If $R$ is non-commutative and $[R, R] \ne R$ then it is easy to see  that
\[
\frac{1}{|[R, R]|}\left(1 + \frac{|[R, R]| - 1}{|R : Z(R)|} \right) \geq \frac{|Z(R)|}{|R|} + \frac{p(|R| - |Z(R)|)}{|R||R|}
\]
with equality  if and only if $|R : [R , R]| = p$. Also,
\[
\frac{1}{|K(R, R)|}\left(1 + \frac{|K(R, R)| - 1}{|R : Z(R)|} \right) \geq \frac{1}{|[R, R]|}\left(1 + \frac{|[R, R]| - 1}{|R : Z(R)|} \right) 
\]
with equality if and only if $K(R, R) = [R, R]$.
Hence, the lower bound obtained in Theorem \ref{newlb1} is better than the lower bounds obtained in   Theorem \ref{theorem001} and Theorem \ref{newlb2}. We conclude the paper noting that Theorem \ref{newlb1} and Theorem \ref{newlb2} are  analogous to \cite[Theorem A]{nY15}  and \cite[Theorem 1]{ND10} respectively.
 


\end{document}